\documentclass[12pt]{article}
\usepackage{amsmath,amsfonts,amsthm,amssymb, mathtools}
\usepackage{mathrsfs, graphicx,color,latexsym, tikz, calc,cite,enumerate,indentfirst}
\usetikzlibrary{shadows}
\usetikzlibrary{patterns,arrows,decorations.pathreplacing}
\usepackage[colorlinks=true,
linkcolor=blue,citecolor=blue,
urlcolor=blue]{hyperref}

\newtheorem{theorem}{Theorem}[section]
\newtheorem{lemma}{Lemma}[section]

\newtheorem{conjecture}{Conjecture}[section]

\newtheorem{claim}{Claim}

\renewcommand\proofname{\it{Proof}}

\allowdisplaybreaks

\title{\bf Proof of Lew's conjecture on the spectral gaps of simplicial complexes}

\author{Xiongfeng Zhan, Xueyi Huang,  Huiqiu Lin\setcounter{footnote}{-1}\footnote{\emph{Email address:} zhanxfmath@163.com (X. Zhan), huangxymath@163.com (X. Huang), huiqiulin@126.com (H. Lin).}\\[2mm]
	\small School of Mathematics, East China University of Science and Technology, \\
	\small  Shanghai 200237, P. R. China
}

\date{}

\begin{document}
	\maketitle
	\begin{abstract}
	As a generalization of graph Laplacians to higher dimensions, the combinatorial Laplacians of simplicial complexes have garnered increasing attention. 	Let $X$ be a simplicial complex on vertex set $V$ of size $n$, and let $X(k)$ denote the set of all $k$-dimensional simplices of $X$. The $k$-th spectral gap $\mu_k(X)$ is the smallest eigenvalue of the reduced $k$-dimensional Laplacian of $X$.  For any $k\geq -1$,  Lew [J. Combin. Theory Ser. A 169 (2020) 105127] established  a lower bound for $\mu_k(X)$:
		 $$\mu_k(X)\geq (d+1)\left(\min_{\sigma\in X(k)}\deg_X(\sigma)+k+1\right)-dn\geq (d+1)(k+1)-dn,$$
	where  $\deg_X(\sigma)$ and $d$ denote the degree of $\sigma$ in $X$ and the  maximal dimension of a missing face of $X$, respectively. In this paper, we identify the unique simplicial complex
	that achieves the lower bound of the $k$-th spectral gap, $(d+1)(k+1)-dn$, for some 
$k$, thereby confirming a conjecture proposed by Lew.
		
		\par\vspace{2mm}
		
		\par\vspace{1mm}
		
	\end{abstract}

	\section{Introduction}\label{section::1}

	An \textit{abstract simplicial complex} $X$ is a collection of  finite sets that is closed under set inclusion, i.e., if $\sigma\in X$ and $\tau\subseteq \sigma$, then $\tau\in X$. A finite set $\sigma$ in $X$ is called a \textit{simplex} of  $X$, and the \textit{dimension} of $\sigma$ is defined as $|\sigma|-1$. Note that the empty set $\varnothing$ is always in $X$, and its dimension is  $-1$. The \textit{dimension} of $X$, denoted by $\dim(X)$, is the largest dimension of all of the simplices in $X$. The \textit{vertex set} of $X$, denoted by $V(X)$, is the union of all $0$-dimensional simplices of $X$. A \textit{subcomplex} of a simplicial complex $X$ is a subset of $X$ that is itself a simplicial complex. 	A \textit{missing face} in $X$ is a subset $\sigma$ of $V$ such that $\sigma\notin X$ but $\tau\in X$ for every proper subset $\tau$ of $\sigma$. The maximal dimension of a missing face of $X$ is denoted by $h(X)$. 
	
	For $k\geq -1$, 
let $C^k(X)$ denote the space of real valued $k$-cochains of $X$, and let $d_k:C^k(X)\rightarrow C^{k+1}(X)$ denote the coboundary operator. Then the 
	\textit{reduced $k$-dimensional Laplacian} of $X$ is defined by \[L_k(X)=\begin{cases}
		d_{-1}^*d_{-1}&\mbox{if $k=-1$},\\
		d_{k-1}d_{k-1}^*+d_{k}^*d_{k}&\mbox{if $k\geq 0$},
	\end{cases}\] 
	where $d_k^*:C^{k+1}(X)\rightarrow C^{k}(X)$ is the adjoint operator of $d_k$. Clearly, $L_k(X)$ is a positive semi-definite operator on $C^k(X)$. 
	 The $k$-th \textit{spectral gap} of $X$, denoted by $\mu_k(X)$, is the smallest eigenvalue of $L_k(X)$.

The theory of graph Laplacians dates back to Kirchhoff \cite{Kir47} in his study of electrical networks, and his celebrated matrix-tree theorem (see, for example, \cite{Kal83}). Over the past half century, there has been a wealth of research on the spectra of graph Laplacians \cite{AM85, Chu96, Fie73, GMS90, GM94, Mer94}. As a generalization of graph Laplacians to higher dimensions, the combinatorial Laplacians of simplicial complexes have attracted increasing attention in recent years. For instance, the spectra of Laplacians have been shown to be integral for several well-known families of simplicial complexes, such as the chessboard complexes \cite{FH98}, the matching complexes \cite{DW02}, the matroid complexes \cite{KRS02}, and the shifted simplicial complexes \cite{DR02}. Among other things, Duval, Klivans, and Martin \cite{DKM09} established a simplicial version of the matrix-tree theorem that counts simplicial spanning trees, weighted by the
squares of the orders of their top-dimensional integral homology groups, in
terms of the Laplacians of simplicial complexes. For further results on the  Laplacians of simplicial complexes, we refer the reader to \cite{ABM05,Den01,Eck44,GW16,HJ13,Lew20,Lew20+,Lew23,Lim20,SBHLJ20,SY20}.
	
Let $X(k)$ denote the set of all $k$-dimensional simplices of $X$. If $\sigma\in X(k)$, then the  \textit{degree} of $\sigma$ in $X$ is defined by
	$\deg_X(\sigma)=|\{\eta\in X(k+1): \sigma\subseteq \eta\}|$. In 2020, Lew \cite{Lew20} established a lower bound for the $k$-th spectral gap $\mu_k(X)$ of $X$ in terms of the minimum degree of $\sigma\in X(k)$  and the maximum dimension of a missing face of $X$.

	\begin{theorem}[{\cite[Theorem 1.1]{Lew20}}]\label{thm::Lew_bound}
		Let $X$ be a simplicial complex on vertex set $V$ of size $n$, with $h(X)=d$. Then for $k\geq -1$, 
		$$\mu_k(X)\geq (d+1)\left(\min_{\sigma\in X(k)}\deg_X(\sigma)+k+1\right)-dn.$$
		In particular, 
		\begin{equation*}
			\mu_k(X)\geq (d+1)(k+1)-dn.
		\end{equation*}
	\end{theorem}
	Let $X$ and $Y$ be two simplicial complexes on disjoint vertex sets. The \textit{join} of $X$ and $Y$ is the simplicial  complex 
	\begin{equation*}
		X*Y=\{\sigma\cup\tau:\sigma\in X, \tau\in Y\}.
	\end{equation*}	
	We  denote by $X*X$ the join of $X$ with a disjoint copy of itself. Also, we denote the  complex $X*X*\cdots*X$ ($k$ times) by $X^{*k}$.	Furthermore, 
	we will denote $X_1*X_2*\cdots*X_s$ by $\mathrm{Join}\{X_i:1\leq i\leq s\}$.

	Let $\Delta_m$ be the complete simplicial complex on $m+1$ vertices. The \textit{$k$-dimensional skeleton} of $\Delta_m$, denoted by $\Delta_m^{(k)}$, is the subcomplex of $\Delta_m$ whose simplices are all the sets $\sigma\in\Delta_m$ with $|\sigma|\leq k+1$. Clearly, $\Delta_m^{(m)}=\Delta_m$. By using the join of two types  of skeletons, Lew \cite{Lew20} constructed an example of  simplicial complex that achieves the lower bound in Theorem \ref{thm::Lew_bound}. 
	\begin{theorem}[{\cite[Proposition 1.5]{Lew20}}]\label{thm::extremal_case}
		Let $Z =\left(\Delta_d^{(d-1)}\right)^{*t}*\Delta_{r-1}$, and let $n = (d + 1)t + r$ be the number of
		vertices of $Z$. Then all missing faces of $Z$ are of dimension $d$,  and 
		\[
		\mu_k(Z)=\begin{cases}
			(d+1)\left(t-\lfloor\frac{k+1}{d}\rfloor\right)+r&\mbox{if $-1\leq k\leq dt-1$},\\
			r&\mbox{if $dt\leq k\leq dt+r-1$},\\
		\end{cases}
		\]
		and 
		\[
		\min_{\sigma\in Z(k)}\deg_Z(\sigma)=\begin{cases}
			n-(k+1)-\lfloor\frac{k+1}{d}\rfloor&\mbox{if $-1\leq k\leq dt-1$},\\
			n-(k+1)-t&\mbox{if $dt\leq k\leq dt+r-1$}.\\
		\end{cases}
		\]
		Therefore, 
		\[\mu_k(Z)=(d+1)\left(\min_{\sigma\in X(k)}\deg_X(\sigma)+k+1\right)-dn\] 
		for all $-1\leq k\leq \dim(Z)=dt+r-1$.
	\end{theorem}
	
		Let $X$ and $Y$ be two simplicial complexes. A \textit{simplicial mapping} of $X$ into $Y$ is a mapping $f:V(X)\rightarrow V(Y)$ such that $f(\sigma)\in Y$ for all $\sigma\in X$. We say that $X$ and  $Y$ are \textit{isomorphic}, denoted by $X\cong Y$, if there is a bijective mapping $f:V(X)\rightarrow V(Y)$ such that both $f$ and $f^{-1}$ are simplicial mappings.
		
In	\cite{Lew20}, Lew also considered characterizing the simplicial complexes that attain the equality $\mu_k(X)=(d+1)(k+1)-dn$ at some dimension $k$, and proposed the following conjecture.
	\begin{conjecture}[{\cite[Conjecture 5.1]{Lew20}}]\label{conj::main}
		Let $X$ be a simplicial complex on vertex set $V$ of size $n$, with $h(X)=d$, such that $\mu_k(X)=(d+1)(k+1)-dn$ for some $k$. Then 
		$$X\cong \left(\Delta_d^{(d-1)}\right)^{*(n-k-1)}*\Delta_{(d+1)(k+1)-dn-1}.$$
		In particular, $\dim(X)=k$.
	\end{conjecture}
	
In the same paper, Lew showed that Conjecture \ref{conj::main} is true for $d=1$. In this paper, we completely confirm  Conjecture \ref{conj::main}.
	
	\begin{theorem}\label{thm::main}
		Conjecture \ref{conj::main} is true for all $d\geq 1$.
	\end{theorem}

	\section{Preliminaries}\label{section::2}
	
	Let $X$ be an abstract simplicial complex on  vertex set $V$ of size $n$. An \textit{ordered simplex} is a simplex with a total order of its vertices. For a simplex $\{v_0, \ldots, v_k\} \in X$ with its vertices ordered by $v_0 \prec v_1 \prec \cdots \prec v_k$, we denote the corresponding ordered simplex by $[v_0, \ldots, v_k]$. For two disjoint ordered simplices $\sigma = [v_0, \ldots, v_k]$ and $\tau = [u_0, \ldots, u_s]$, we denote by $[\sigma, \tau] = [v_0, \ldots, v_k, u_0, \ldots, u_s]$ the ordered union of $\sigma$ and $\tau$.

For two ordered simplices $\tau$ and $\sigma$ with $\tau \subseteq \sigma$, we use $(\sigma : \tau)$ to represent the sign of the permutation on $\sigma$ that maps the ordered simplex $\sigma$ to the ordered simplex $[\sigma \setminus \tau, \tau]$, where the order on $\sigma \setminus \tau$ is the one induced by the order on $\sigma$.

In what follows, we fix a total order $\prec$ on the vertex set $V$ of $X$. Then, for every $k \geq -1$, we can identify $X(k)$ with the set of all the ordered $k$-dimensional simplices in $X$ with order induced by $\prec$.

	A \textit{simplicial $k$-cochain} $\varphi$ is a skew-symmetric function on all ordered $k$-dimensional simplices. More specifically, $\varphi$ is a simplicial $k$-cochain if for any two ordered $k$-dimensional simplices $\sigma,\tau$ in $X$ that are equal as sets, we have $\varphi(\sigma)=(\sigma:\tau)\varphi(\tau)$. 
	
	For a simplex $\sigma\in X$, we say that $\tau$ is a \textit{face} of $\sigma$ if $\tau\subseteq\sigma$. Let $\sigma(k)$ denote the set of all $k$-dimensional faces of $\sigma$. Clearly, $\sigma(k)\subseteq X(k)$. 	Let  $C^k(X)$ be the space of real valued simplicial $k$-cochains of $X$. The \textit{coboundary operator} $d_k:C^k(X)\rightarrow C^{k+1}(X)$ is the linear operator defined by letting
	\begin{equation*}
		d_k\varphi(\sigma)=\sum_{\tau\in\sigma(k)}(\sigma:\tau)\varphi(\tau)
	\end{equation*}
	for any $\varphi\in C^k(X)$ and   $\sigma\in X(k+1)$. For $k=-1$, we set $C^{-1}(X)=\mathbb{R}$. Then  $d_{-1}a(v)=a$ for every $a\in\mathbb{R}$ and $v\in V$.  Furthermore, we define an inner product on $C^k(X)$ by 
	\begin{equation*}
		\left\langle\varphi,\psi\right\rangle=\sum_{\sigma\in X(k)}\varphi(\sigma)\psi(\sigma)
	\end{equation*}
for any $\varphi,\psi\in C^k(X)$. Let $d_k^*:C^{k+1}(X)\rightarrow C^k(X)$ be the adjoint of $d_k$ with respect to this inner product. The \textit{reduced $k$-dimensional Laplacian}  of $X$ is given by
	\[L_k=L_k(X)=\begin{cases}
		d_{-1}^*d_{-1}&\mbox{if $k=-1$},\\
		d_{k-1}d_{k-1}^*+d_{k}^*d_{k}&\mbox{if $k\geq 0$}.
	\end{cases}\] 
	Note that $L_k$ is a positive semi-definite operator from $C^k(X)$ to itself. In particular, for $k=-1$, $L_{-1}$ is a mapping from $\mathbb{R}$ to $\mathbb{R}$, and  $L_{-1}(a)=|V|=n$ for all $a\in \mathbb{R}$. The $k$-th \textit{spectral gap} of $X$, denoted by $\mu_k(X)$, is the smallest eigenvalue of $L_k$.

For any $\sigma\in X(k)$, we define the real valued function $1_\sigma$ by letting 
	\[
	1_\sigma(\tau)=\begin{cases}
		(\sigma:\tau) &\mbox{if $\sigma=\tau$ (as sets)},\\
		0&\mbox{otherwise.}
	\end{cases}
	\]
Then the set $\{1_\sigma:\sigma\in X(k)\}$ forms a basis of the space $C^k(X)$. With respect to this basis, we still use $L_k$ to denote the matrix
	representation of the operator $L_k$. Additionally, the entry of the matrix 
	$L_k$ at index $(1_\sigma,1_\tau)$ is denoted by $L_k(\sigma, \tau)$.

	For any $\sigma,\tau\in X(k)$, we write  $\sigma\sim\tau$ if  $|\sigma\cap \tau|=k$ and $\sigma\cup\tau\notin X(k+1)$. Then the value of  $L_k(\sigma,\tau)$ can be determined explicitly based on whether $\sigma = \tau$, $\sigma \sim \tau$, or neither of the two situations occurs.
	
	\begin{lemma}[{\cite{DR02,Gol02}}]\label{lem::Matrix_Rep}
		Let $k\geq 0$. Then
		\begin{equation*}
			L_k(\sigma,\tau)=\begin{cases} 
			\deg_X(\sigma)+k+1 & \mbox{if $\sigma=\tau$}, \\
				(\sigma:\sigma\cap\tau)\cdot(\tau:\sigma\cap\tau) &\mbox{if $\sigma\sim\tau$}, \\
				0 & \mbox{otherwise}.
			\end{cases}   
		\end{equation*}
	\end{lemma}

Let $H_k$ denote the matrix with rows indexed by $X(k)$ and columns indexed by $E_k:=\{\{\sigma,\tau\}:\sigma,\tau\in X(k),\sigma\sim\tau\}$ defined by
	\begin{equation*}
		H_k(\sigma,\{\eta,\tau\})=\begin{cases} 
			(\sigma:\eta\cap\tau) & \mbox{if}~\sigma\in\{\eta,\tau\}, \\
			0 & \mbox{otherwise}.
		\end{cases}   
	\end{equation*}
	Let $D_k$ be the diagonal matrix indexed by $X(k)$ with diagonal elements \begin{equation}\label{equ::diag}
		D_k(\sigma,\sigma)=2(k+1)+(k+2)\deg_X(\sigma)-\sum_{\tau\in \sigma(k-1)}\deg_X(\tau).
	\end{equation}
Then the matrix $L_k$ can be expressed as the sum of  $D_k$ and $H_kH_k^T$.
	\begin{lemma}[{\cite[Remark]{Lew20}}]\label{lem::L_k}
		\[L_k=D_k+H_kH_k^T.\]
	\end{lemma}
	
Observe that the coboundary operator $d_k:C^k(X)\rightarrow C^{k+1}(X)$ satisfies the relation $d_kd_{k-1}=0$. Consequently,  we have  $\mathrm{Im}(d_{k-1})\subseteq \mathrm{Ker}(d_k)$. If we consider  $\mathrm{Im}(d_{k-1})$ and $\mathrm{Ker}(d_k)$ as groups, then the quotient group  $\widetilde{H}^k(X;\mathbb{R}):=\mathrm{Ker}(d_k)/\mathrm{Im}(d_{k-1})$ is called the $k$-th \textit{reduced cohomology group} of $X$ with real coefficients. The following theorem establishes a relationship between the cohomology groups and the Laplacian operators of $X$.
	
	\begin{theorem}[Simplicial Hodge theorem, \cite{Eck44}]\label{thm::Hodge}
		\[\widetilde{H}^k(X;\mathbb{R})\cong \mathrm{Ker}(L_k).\]
	\end{theorem}	
	
	\section{Geometric realization}\label{section::3}
	
	Let $\boldsymbol{v}_0,\boldsymbol{v}_1,\ldots,\boldsymbol{v}_s$ be points in $\mathbb{R}^d$. We call them \textit{affinely dependent} if there are real numbers $\alpha_0,\alpha_1,\ldots,\alpha_s$, not all of them $0$, such
	that $\sum_{i=0}^s\alpha_i\boldsymbol{v}_i=0$ and $\sum_{i=0}^s\alpha_i=0$. Otherwise, $\boldsymbol{v}_0,\boldsymbol{v}_1,\ldots,\boldsymbol{v}_s$ are called \textit{affinely independent}.
	
	Let $F$ be a finite affinely independent set in $\mathbb{R}^d$. The (geometric) \textit{simplex} $\sigma$ spanned by $F$ is defined as the convex hull \[\mathrm{conv}(F):=\left\{\sum_{\boldsymbol{v}\in F}\alpha_{\boldsymbol{v}}\boldsymbol{v}:\sum_{\boldsymbol{v}\in F}\alpha_{\boldsymbol{v}}=1~\mbox{and}~\alpha_{\boldsymbol{v}}\geq 0~\mbox{for all $\boldsymbol{v}\in F$}\right\}.\]
	The points of $F$ are called the \textit{vertices} of $\sigma$, and the \textit{dimension} of $\sigma$ is $\dim(\sigma)=|F|-1$. Thus every $k$-dimensional simplex has $k+1$ vertices. The convex hull of an arbitrary subset of vertices of a simplex $\sigma$ is called a \textit{face} of $\sigma$. 
	
	A nonempty family $\Delta$ of simplices is a \textit{geometric simplicial complex} if the following two conditions hold:
	\begin{enumerate}[(1)]
		\item each face of any simplex $\sigma\in \Delta$ is also a simplex of $\Delta$;
		\item the intersection $\sigma_1\cap \sigma_2$ of any two simplices $\sigma_1,\sigma_2\in \Delta$ is a face of both $\sigma_1$ and $\sigma_2$.
	\end{enumerate}

Let $\Delta$ be a geometric simplicial complex.	The union of all simplices in $\Delta$ is the \textit{polyhedron} of $\Delta$ and is denoted by $\|\Delta\|$.
	The \textit{dimension} of  $\Delta$ is defined as $\dim(\Delta):=\max\{\dim(\sigma):\sigma\in\Delta\}$. The vertex set of $\Delta$, denoted by $V(\Delta)$, is the union of the vertex sets of all simplices of $\Delta$. A \textit{subcomplex} of  $\Delta$ is a subset of $\Delta$
	that is itself a geometric simplicial complex. 
	
	Note that the set of all faces of a (geometric) simplex is a geometric simplicial complex.  A geometric simplicial complex consisting of all faces of an arbitrary $k$-dimensional simplex (including the simplex itself) will be denoted by $\sigma^k$.

	Each geometric simplicial complex $\Delta$ determines an abstract simplicial complex $X$. The vertices of the abstract simplicial complex $X$ are all vertices of the simplices of $\Delta$, and the sets in the abstract simplicial complex $X$ are just the vertex sets of the simplices of $\Delta$.  In this situation, we call $\Delta$ a \textit{geometric realization} of $X$, and the polyhedron of $\Delta$ is also referred to as a \textit{polyhedron} of $X$. 
	
	Conversely, every finite  abstract simplicial complex $X$ has a geometric realization $\Delta$.  Let $V=V(X)$ and $n=|V|$.  Let us identify $V$ with the vertex set of  $\sigma^{n-1}$. We define a subcomplex $\Delta$  of $\sigma^{n-1}$ by $\Delta := \{\mathrm{conv}(F): F \in  X\}$. This is a geometric simplicial complex, and its associated abstract simplicial complex is $X$. Thus every abstract simplicial complex on $n$ vertices has a geometric realization in  $\mathbb{R}^{n-1}$.
	
	Let $\mathcal{X}$ and $\mathcal{Y}$ be two topological spaces. A mapping $f:\mathcal{X}\rightarrow \mathcal{Y}$ is a \textit{homeomorphism} if $f$ is bijective and both $f$ and $f^{-1}$ are continuous. 
	We say that $\mathcal{X}$ and $\mathcal{Y}$ are  \textit{homeomorphic}, denoted by $\mathcal{X}\cong\mathcal{Y}$, if there is a homeomorphism between $\mathcal{X}$ and $\mathcal{Y}$. By elementary algebraic topology, the polyhedron of an abstract simplicial complex $X$ is unique up to homeomorphism (cf.  \cite[p. 14]{Mat03}), and so we can denote it by $\|X\|$. An abstract simplicial complex $X$ is called a \textit{triangulation} of a  topological space $\mathcal{X}$ if  its polyhedron $\|X\|$ is homeomorphic to  $\mathcal{X}$.

Let $\boldsymbol{x}=(x_1,\ldots,x_d)^T$ be a point in $\mathbb{R}^d$. We define the \textit{norm} of $\boldsymbol{x}$ as  \[\|\boldsymbol{x}\|=\left[\sum_{i=1}^dx_i^2\right]^{\frac{1}{2}}.\] The $(d-1)$-\textit{sphere}, denoted by $S^{d-1}$, is the set of all points $\boldsymbol{x}$ of $\mathbb{R}^d$ for which $\|\boldsymbol{x}\|=1$. Recall that $\Delta_d^{(d-1)}$ is the $(d-1)$-dimensional skeleton of the complete simplicial complex $\Delta_d$ on $d+1$ vertices. The geometric realization of $\Delta_d^{(d-1)}$ is the \textit{boundary} of $\sigma^d$, that is, the subcomplex of $\sigma^d$ obtained by deleting the single $d$-dimensional simplex (but retaining all of its proper faces). According to  elementary algebraic topology,  the polyhedron of the boundary of $\sigma^d$ is homeomorphic to $S^{d-1}$, and so we have the following result.
	
	\begin{lemma}[{\cite[p.  10]{Mat03}}]\label{lem::skeleton}
		Let $d$ be a positive integer. Then  \[\|\Delta_d^{(d-1)}\|\cong S^{d-1}.\]  
	\end{lemma}

	The \textit{join} of two topological spaces  $\mathcal{X}$ and $\mathcal{Y}$, denoted by  $\mathcal{X}*\mathcal{Y}$, is the quotient space $\mathcal{X}\times \mathcal{Y}\times [0,1]/\approx$, where the equivalence relation $\approx$ is given by $(x_1,y,0)\approx(x_2,y,0)$ for all $x_1,x_2\in \mathcal{X}$, $y\in \mathcal{Y}$ and $(x,y_1,1)\approx(x,y_2,1)$ for all $x\in X$, $y_1,y_2\in Y$. As in Section \ref{section::1}, we denote the topological space $\mathcal{X}*\mathcal{X}*\cdots*\mathcal{X}$ ($k$ times) by $\mathcal{X}^{*k}$. 
	
	The following two lemmas respectively determine the structure of the join of two spheres and the structure of the join of polyhedrons of two abstract simplicial complexes.

	\begin{lemma}[{\cite[p. 19, Exercise 18]{Hat02}}]\label{lem::SnSm}
	Let $n$ and $m$ be two non-negative integers. Then 
	\[S^n*S^m\cong S^{n+m+1}.\]
	\end{lemma}
	
	\begin{lemma}[{\cite[p. 77, Exercise 3]{Mat03}}]\label{lem::join}
		Let $X$ and $Y$ be two finite abstract simplicial complexes. Then
		\begin{equation*}
			\|X\|*\|Y\|\cong\|X*Y\|.
		\end{equation*}
	\end{lemma}

	\section{Key lemmas}\label{section::4}

	The Ger\v{s}gorin circle theorem is a basic tool in matrix theory. Here we use the form that can be derived from its proof.
	
	\begin{theorem}[Ger\v{s}gorin circle theorem, {\cite[p. 388]{HJ12}}]\label{thm::Ger_circle}
		Let $A=(a_{ij})\in\mathbb{C}^{n\times n}$, and let $\lambda\in\mathbb{C}$ be an eigenvalue of $A$. Let $\boldsymbol{x}=(x_1,\ldots,x_n)^T\in \mathbb{C}^n$ be an eigenvector of $A$ with respect to $\lambda$, and let $i\in [n]$ be an index such that $|x_i|=\max_{1\leq j\leq n}|x_j|$. Then  
		\begin{equation*}
			|\lambda-a_{ii}|\leq\sum_{j\neq i}|a_{ij}|.
		\end{equation*}
		Moreover, if the equality holds, then $|x_{j}|=|x_{i}|$ whenever $a_{ij}\neq 0$, and all complex numbers in $\{a_{ij}x_j: j\neq i,a_{ij}\neq 0\}$ have the same argument.
	\end{theorem}

	Let $A$ be a real symmetric matrix, and let $\lambda_{\min}(A)$ denote the smallest eigenvalue of $A$. By using the Rayleigh quotient theorem, we can deduce the following result.
	
	\begin{lemma}\label{lem::principal}
		Let $A$ be a real symmetric matrix with order $n$, and let $B$ be a principal submatrix of $A$ with order $m$ ($m<n$). Suppose that the columns of $B$ are indexed by  $\mathcal{I}_B\subseteq [n]$.  Then  $\lambda_{\min}(A)=\lambda_{\min}(B)$ if and only if $A$ has an eigenvector $\boldsymbol{x}=(x_1,\ldots,x_n)^T\in \mathbb{R}^n$ with respect to $\lambda_{\min}(A)$ such that $x_i=0$ for all $i\notin \mathcal{I}_B$. 
	\end{lemma}
	\begin{proof}
		By the Rayleigh quotient theorem (cf. \cite[Theorem 4.2.2]{HJ12}), 
		\begin{equation*}
			\lambda_{\min}(A)=\min_{\boldsymbol{x}\in \mathbb{R}^n\setminus\{\boldsymbol{0}\}}\frac{\boldsymbol{x}^TA\boldsymbol{x}}{\boldsymbol{x}^T\boldsymbol{x}}, 
		\end{equation*}   
		where the minimum value is achieved if and only if $\boldsymbol{x}$ is an eigenvector of $A$ with respect to $\lambda_{\min}(A)$. Let  $U$ be the matrix obtained from the identity matrix $I_n$ by deleting all columns indexed by $[n]\setminus\mathcal{I}_B$. Then we have $B=U^TAU$, and again by the Rayleigh quotient theorem, 
		\begin{equation*}
			\begin{aligned}
				\lambda_{\min}(B)&=\lambda_{\min}(U^TAU)
				=\min_{\boldsymbol{y}\in \mathbb{R}^m\setminus\{\boldsymbol{0}\}}\frac{\boldsymbol{y}^T(U^TAU)\boldsymbol{y}}{\boldsymbol{y}^T\boldsymbol{y}}
				=\min_{\boldsymbol{y}\in \mathbb{R}^m\setminus\{\boldsymbol{0}\}}\frac{(U\boldsymbol{y})^TAU\boldsymbol{y}}{(U\boldsymbol{y})^TU\boldsymbol{y}}\\
				&=\min_{\boldsymbol{x}\in \mathbb{H}}\frac{\boldsymbol{x}^TA\boldsymbol{x}}{\boldsymbol{x}^T\boldsymbol{x}}
				\geq \min_{x\in \mathbb{R}^n\setminus\{\boldsymbol{0}\}}\frac{\boldsymbol{x}^TA\boldsymbol{x}}{\boldsymbol{x}^T\boldsymbol{x}}=\lambda_{\min}(A),
			\end{aligned}
		\end{equation*}
		where $\mathbb{H}$ is the set of vectors $\boldsymbol{x}\in \mathbb{R}^n\setminus\{\boldsymbol{0}\}$ such that $x_i=0$ for all $i\notin \mathcal{I}_B$. Therefore, we conclude that $\lambda_{\min}(B)=\lambda_{\min}(A)$ if and only if $A$ has an eigenvector $\boldsymbol{x}\in \mathbb{H}$ with respect to $\lambda_{\min}(A)$. The result follows.
	\end{proof}

	Let $X$ be an abstract simplicial complex on vertex set $V$, and let $\sigma\in X$.  The \textit{link} of $\sigma$ in $X$ is defined as
	$$\mathrm{lk}(X,\sigma)=\{\tau\in X:\tau\cup\sigma\in X,\tau\cap\sigma=\varnothing\}.$$ 
	For $U\subseteq V$, we denote by $$X[U]=\{\sigma\in X:\sigma\subseteq U\}$$
	the subcomplex of $X$ induced by $U$. Let $k\geq 0$ and $\sigma\in X(k)$. For any $v\in V\setminus \sigma$, we define
	\[N_{\sigma}(v)=\{\tau\in\sigma(k-1):v\in\mathrm{lk}(X,\tau)\}~\mbox{and}~M_{\sigma}(v)=\{\sigma\setminus\tau:\tau\in N_{\sigma}(v)\}.\] Clearly,  $|N_{\sigma}(v)|=|M_{\sigma}(v)|$.

	In \cite{Lew20}, Lew obtained the following inequality concerning the sum of degrees of simplices in $X$, which is pivotal in the proof of the main result.
	
	\begin{lemma}[{\cite[Lemma 1.4]{Lew20}}]\label{lem::sum_deg_1}
		Let $X$ be a simplicial complex on vertex set $V$ of size $n$, with $h(X)=d$. Let $k\geq 0$ and $\sigma\in X(k)$. Then 
		$$\sum_{\tau\in \sigma(k-1)}\deg_X(\tau)-(k-d+1)\deg_X(\sigma)\leq dn-(d-1)(k+1).$$
	\end{lemma}

	In order to prove Lemma \ref{lem::sum_deg_1}, Lew \cite{Lew20} established the following equality.
	
	\begin{lemma}[{\cite[Claim 3.1]{Lew20}}]\label{lem::sum_deg_2}
		Let $X$ be a simplicial complex on vertex set $V$. Let $k\geq 0$ and $\sigma\in X(k)$. Then
		$$\sum_{\tau\in \sigma(k-1)}\deg_X(\tau)=(k+1)(\deg_X(\sigma)+1)+\sum_{v\in V\setminus \sigma,
			v\notin \mathrm{lk}(X,\sigma)}|N_{\sigma}(v)|.$$
	\end{lemma}
	
	In the proof of Lemma \ref{lem::sum_deg_1}, Lew \cite{Lew20} also  established an upper bound for the cardinality of $N_\sigma(v)$ (or $M_\sigma(v)$) under the assumption that $v\notin \mathrm{lk}(X,\sigma)$. Here, we present an alternative proof and further determine the conditions under which this upper bound is achieved.
	
	\begin{lemma}\label{lem::count_d}
		Let $X$ be a simplicial complex on vertex set $V$, with $h(X)=d$. Let $\sigma\in X(k)$ and $v\in V\setminus \sigma$. If $v\notin \mathrm{lk}(X,\sigma)$, then
		\[|N_{\sigma}(v)|=|M_{\sigma}(v)|\leq d.\]
		Furthermore, if the equality holds, then 
		\[X[\{v\}\cup M_{\sigma}(v)]\cong \Delta_d^{(d-1)}~\mbox{and}~X[\{v\}\cup \sigma]\cong \Delta_d^{(d-1)}*\Delta_{k-d}.\]
	\end{lemma}
	\begin{proof}
		Since $|N_\sigma(v)|=|M_\sigma(v)|\leq |\sigma(k-1)|=k+1$, there is nothing to prove if $d>k+1$. Thus we can suppose  $d\leq k+1$. By contradiction, assume that $|M_{\sigma}(v)|>d$. Let $W=\{w_1,w_2,\ldots,w_{d+1}\}$ be a $(d+1)$-subset of $\{v\}\cup \sigma$.  If $W\subseteq\sigma$, then $W\in X$. If $W\not\subseteq\sigma$, then $v\in W$, and $W$ would miss some element in $M_{\sigma}(v)$ due to $|M_\sigma(v)|\geq d+1>d=|W\cap \sigma|$. Hence, there exists some  $\tau\in  N_{\sigma}(v)$ such that $W\subseteq \{v\}\cup \tau$. As $v\in \mathrm{lk}(X,\tau)$, we have $\{v\}\cup\tau\in X$, which implies that  $W\in X$. Hence, every $(d+1)$-subset of $\{v\}\cup \sigma$ is a simplex in $X$. Then we see that every $(d+2)$-subset of $\{v\}\cup \sigma$ must belong to $X$, since otherwise it would be a missing face of dimension $d+1$, contrary to $h(X)=d$. Continuing this way, we assert that every $\ell$-subset of $\{v\}\cup\sigma$ with $\ell\geq d+1$ is a simplex in $X$. In particular, $\{v\}\cup\sigma\in X$. However, this is impossible because $v\notin \mathrm{lk}(X,\sigma)$. Therefore, we conclude that $|N_{\sigma}(v)|=|M_{\sigma}(v)|\leq d$. This proves the first part of the lemma. 
		
		Now assume that $|N_\sigma(v)|=|M_{\sigma}(v)|=d$. Clearly, $d\leq k+1$. Let $W$ be a subset of $\{v\}\cup \sigma$. If $W\not\supseteq \{v\}\cup M_{\sigma}(v)$, then  $W\subseteq \sigma$ or $W$ misses some element in $M_\sigma(v)$. In both cases, as above, we can deduce that $W\in X$. Hence, every subset of $\{v\}\cup \sigma$ that does not contain $\{v\}\cup M_{\sigma}(v)$ is a simplex in  $X$. Furthermore, we claim that $\{v\}\cup M_{\sigma}(v)\notin X$. If not, then all $(d+1)$-subsets of $\{v\}\cup \sigma$ would belong to $X$ because $|\{v\}\cup M_{\sigma}(v)|=d+1$. In this case, as above, we can deduce that $\{v\}\cup \sigma\in X$, contrary to $v\notin \mathrm{lk}(X,\sigma)$. Also, from $\{v\}\cup M_{\sigma}(v)\notin X$ we  immediately obtain that $W\notin X$ whenever $W\supseteq \{v\}\cup M_{\sigma}(v)$. Therefore, we conclude that $W\notin X$ if and only if $W\supseteq \{v\}\cup M_{\sigma}(v)$. Consequently, we have $X[\{v\}\cup M_{\sigma}(v)]\cong \Delta_d^{(d-1)}$ and $X[\{v\}\cup \sigma]\cong \Delta_d^{(d-1)}*\Delta_{k-d}$, as desired.
		
		We complete the proof.
	\end{proof}   
	
	\section{Proof of Theorem \ref{thm::main}}\label{section::5}	
	Now we are in a position to give the proof of Theorem \ref{thm::main}.
	\renewcommand\proofname{\it{Proof of Theorem \ref{thm::main}}} 
	\begin{proof}
		Let $L_k$ be the matrix representation of the reduced $k$-dimensional Laplacian of $X$, as described in Lemma \ref{lem::Matrix_Rep}. Furthermore, let   $\boldsymbol{x}=(x_\sigma)_{\sigma\in X(k)}\in \mathbb{R}^{|X(k)|}$ be an eigenvector of $L_k$ corresponding to the $k$-th spectral gap $\mu_k(X)=(d+1)(k+1)-dn$. Without loss of generality, we may assume that $\max_{\sigma\in X(k)}|x_\sigma|=1$. Let $\sigma_0\in X(k)$ be such that $|x_{\sigma_0}|=1$. By Theorem \ref{thm::Ger_circle}, Lemma \ref{lem::Matrix_Rep}, Lemma \ref{lem::sum_deg_2} and Lemma \ref{lem::count_d}, we have
		\begin{align*}
			\mu_k(X)&\geq L_k(\sigma_0,\sigma_0)-\sum_{\sigma\in X(k),\sigma\neq \sigma_0}|L_k(\sigma_0,\sigma)|
			\\	
			&= \deg_X(\sigma_0)+k+1-|\{\sigma\in X(k):|\sigma_0\cap \sigma|=k,\sigma_0\cup\sigma\notin X(k+1)\}|
			\\
			&= \deg_X(\sigma_0)+k+1-\sum_{\tau\in \sigma_0(k-1)}|\{v\in V\setminus \sigma_0:v\in\mathrm{lk}(X,\tau),v\notin \mathrm{lk}(X,\sigma_0)\}|
			\\
				&= \deg_X(\sigma_0)+k+1-\sum_{\tau\in \sigma_0(k-1)}(\deg_X(\tau)-1-\deg_X(\sigma_0))\\
			&=(k+2)\deg_X(\sigma_0)+2(k+1)-\sum_{\tau\in \sigma_0(k-1)}\deg_X(\tau)\\
			&=(k+2)\deg_X(\sigma_0)+2(k+1)-\left((k+1)(\deg_X(\sigma_0)+1)+\sum_{v\in V\setminus \sigma_0,
				v\notin \mathrm{lk}(X,\sigma_0)}|N_{\sigma_0}(v)|\right)\\
			&=\deg_X(\sigma_0)+k+1-\sum_{v\in V\setminus \sigma_0,
				v\notin \mathrm{lk}(X,\sigma_0)}|N_{\sigma_0}(v)|\\
			&\geq \deg_X(\sigma_0)+k+1-\sum_{v\in V\setminus \sigma_0,
				v\notin \mathrm{lk}(X,\sigma_0)}d\\
			&= \deg_X(\sigma_0)+k+1-(n-k-1-\deg_X(\sigma_0))d\\	&=(d+1)(\deg_X(\sigma_0)+k+1)-dn\\
			&\geq (d+1)(k+1)-dn.
		\end{align*}
Combining this inequality with the assumption that  $\mu_k(X)=(d+1)(k+1)-dn$, we obtain the following three facts: 
		\begin{enumerate}[(F1)]
			\item $\mu_k(X)=L_k(\sigma_0,\sigma_0)-\sum_{\sigma\in X(k),\sigma\neq \sigma_0}|L_k(\sigma_0,\sigma)|$;
			\item $\deg_X(\sigma_0)=0$, i.e., $v\notin \mathrm{lk}(X,\sigma_0)$ for every $v\in V\setminus \sigma_0$;
			\item $|N_{\sigma_0}(v)|=|M_{\sigma_0}(v)|=d$ for every $v\in V\setminus \sigma_0$.
		\end{enumerate}
		Using the facts (F1)--(F3), we proceed to deduce the following three claims. 
		
		\begin{claim}\label{claim::1}
			For every $v\in V\setminus \sigma_0$, we have  
			\[X[\{v\}\cup M_{\sigma_0}(v)]\cong \Delta_d^{(d-1)}~\mbox{and}~X[\{v\}\cup V(\sigma_0)]\cong \Delta_d^{(d-1)}*\Delta_{k-d}.\] 
		\end{claim}
		\renewcommand\proofname{\it{Proof of Claim \ref{claim::1}}} 
		\begin{proof}
			The result follows from Lemma \ref{lem::count_d}, (F2) and (F3) immediately.
		\end{proof}
		\begin{claim}\label{claim::2}
			Let $\sigma,\eta\in X(k)$ be such that $|\sigma\cap\sigma_0|=|\eta\cap\sigma_0|=|\sigma\cap\eta|=k$. Then
			\[	(\sigma_0:\sigma_0\cap\sigma)\cdot(\sigma:\sigma_0\cap\sigma)\cdot (\sigma_0:\sigma_0\cap\eta)\cdot(\eta:\sigma_0\cap\eta)\cdot(\sigma:\sigma\cap\eta)\cdot(\eta:\sigma\cap\eta)=-1.
			\]
		\end{claim}
		\renewcommand\proofname{\it{Proof of Claim \ref{claim::2}}} 
		\begin{proof}
			By Theorem \ref{thm::Ger_circle} and (F1), we see that  $|x_{\sigma}|=|x_{\sigma_0}|=1$ for all $\sigma\in X(k)\setminus \{\sigma_0\}$ satisfying $L_k(\sigma_0,\sigma)\neq 0$, i.e., $\sigma\sim \sigma_0$, and the numbers in the set $\{L_k(\sigma_0,\sigma)x_\sigma: \sigma\in X(k)\setminus \{\sigma_0\},\sigma\sim \sigma_0\}$ are either all positive or all negative.
			Recall that $L_k(\sigma_0,\sigma)=(\sigma_0:\sigma_0\cap\sigma)\cdot(\sigma:\sigma_0\cap\sigma)\in \{1,-1\}$ whenever $\sigma\sim \sigma_0$. 
			Thus there are only two possibilities: 
			\begin{enumerate}[(i)]
				\item $x_{\sigma}=L_k(\sigma_0,\sigma)=(\sigma_0:\sigma_0\cap\sigma)\cdot(\sigma:\sigma_0\cap\sigma)$ for all $\sigma\in X(k)\setminus \{\sigma_0\}$ with $\sigma\sim\sigma_0$;
				\item $x_{\sigma}=-L_k(\sigma_0,\sigma)=-(\sigma_0:\sigma_0\cap\sigma)\cdot(\sigma:\sigma_0\cap\sigma)$ for all $\sigma\in X(k)\setminus \{\sigma_0\}$ with $\sigma\sim\sigma_0$.
			\end{enumerate}

			In what follows, we only prove the result for (i), since the proof for (ii) is very similar. By Lemma \ref{lem::sum_deg_1} and \eqref{equ::diag}, for any $\sigma\in X(k)$, we have
		\begin{equation}\label{equ::diag1}
				\begin{aligned}
					D_k(\sigma,\sigma)&=2(k+1)+(k+2)\deg_X(\sigma)-\sum_{\tau\in \sigma(k-1)}\deg_X(\tau)\\
					&\geq (d+1)(\deg_X(\sigma)+k+1)-dn\\
					&\geq (d+1)(k+1)-dn=\mu_k(X).
				\end{aligned} 
			\end{equation}
Combining this with Lemma \ref{lem::L_k}, we obtain
			\[
		\mu_k(X)\boldsymbol{x}^T\boldsymbol{x}=\boldsymbol{x}^TL_k\boldsymbol{x}=\boldsymbol{x}^TD_k\boldsymbol{x}+\boldsymbol{x}^TH_kH_k^T\boldsymbol{x}\geq\mu_k(X)\boldsymbol{x}^T\boldsymbol{x}+\boldsymbol{x}^TH_kH_k^T\boldsymbol{x}\geq \mu_k(X)\boldsymbol{x}^T\boldsymbol{x},\]
			implying that  $\boldsymbol{x}^TD_k\boldsymbol{x}=\mu_k(X)\boldsymbol{x}^T\boldsymbol{x}$ and $H_k^T\boldsymbol{x}=0$. 
			If $x_\sigma\neq 0$, then we must have $D_k(\sigma,\sigma)=\mu_k(X)$, and hence
			$\deg_X(\sigma)=0$ by \eqref{equ::diag1}. Therefore, we assert that $\deg_X(\sigma)=0$ for all $\sigma\in X(k)\setminus\{\sigma_0\}$ with  $\sigma\sim \sigma_0$. 
			
			Now consider  $\sigma,\eta\in X(k)$ with $|\sigma\cap\sigma_0|=|\eta\cap\sigma_0|=|\sigma\cap\eta|=k$. According to (F2),  $\deg_X(\sigma_0)=0$. Consequently,  $\sigma\cup \sigma_0\notin X(k+1)$ and $\eta\cup \sigma_0\notin X(k+1)$, that is, $\sigma\sim \sigma_0$ and $\eta\sim \sigma_0$. Then from (i) we obtain 
			\begin{equation}\label{equ::entry}
				\left\{\begin{aligned} x_{\sigma}&=L_k(\sigma_0,\sigma)=(\sigma_0:\sigma_0\cap\sigma)\cdot(\sigma:\sigma_0\cap\sigma)\neq 0,\\
					x_\eta&=L_k(\sigma_0,\eta)=(\sigma_0:\sigma_0\cap\eta)\cdot(\eta:\sigma_0\cap\eta)\neq 0,
				\end{aligned}\right.
			\end{equation}
			which implies that $\deg_X(\sigma)=\deg_X(\eta)=0$. Thus $\sigma\cup \eta\notin X(k+1)$, and it follows that  $\sigma\sim \eta$. Therefore,  $H_k^T\boldsymbol{x}$ has a component corresponding to $\{\sigma,\eta\}$, which is  equal to $x_{\sigma}\cdot(\sigma:\sigma\cap\eta)+x_{\eta}\cdot(\eta:\sigma\cap\eta)$. Since   $H_k^T\boldsymbol{x}=0$, we get
			\begin{equation*}
				x_{\sigma}\cdot(\sigma:\sigma\cap\eta)=- x_{\eta}\cdot(\eta:\sigma\cap\eta),
			\end{equation*}
			or equivalently,
			\begin{equation*}
				x_{\sigma}\cdot(\sigma:\sigma\cap\eta)\cdot x_{\eta}\cdot(\eta:\sigma\cap\eta)=-1.
			\end{equation*}
			Combining this with \eqref{equ::entry}, we obtain the result immediately.
		\end{proof}
		\begin{claim}\label{claim::3}	
			$M_{\sigma_0}(u)\cap M_{\sigma_0}(v)=\varnothing$ for  any two distinct  $u,v\in V\setminus \sigma_0$.  
		\end{claim} 
		\renewcommand\proofname{\it{Proof of Claim \ref{claim::3}}} 
		\begin{proof}
			Let $u,v$ be two distinct vertices in $V\setminus \sigma_0$.	 By contradiction, assume that $w\in M_{\sigma_0}(u)\cap M_{\sigma_0}(v)$. We fixed a total order $\prec$ on $V$ such that $w\prec u\prec v\prec z$ for all $z\in\sigma_0\setminus\{w\}$.
			Let $\sigma=\{u\}\cup(\sigma_0\setminus\{w\})$ and $\eta=\{v\}\cup(\sigma_0\setminus\{w\})$. We have $\sigma,\eta\in X(k)$ and  $|\sigma\cap\sigma_0|=|\eta\cap\sigma_0|=|\sigma\cap\eta|=k$. However, by simple observation,
			\begin{equation*}
	(\sigma_0:\sigma_0\cap\sigma)\cdot(\sigma:\sigma_0\cap\sigma)\cdot (\sigma_0:\sigma_0\cap\eta)\cdot(\eta:\sigma_0\cap\eta)\cdot(\sigma:\sigma\cap\eta)\cdot(\eta:\sigma\cap\eta)=1^6=1,
			\end{equation*}
			which is impossible by Claim \ref{claim::2}. Therefore,  $M_{\sigma_0}(u)\cap M_{\sigma_0}(v)=\varnothing$.
		\end{proof}

		Let $A=\sigma_0\setminus\left(\cup_{v\in V\setminus \sigma_0}M_{\sigma_0}(v)\right)$, and let $r:=|A|$. Note that $X[A]\cong \Delta_{r-1}$. According to Claim \ref{claim::3} and (F3), we have \[r=|A|=k+1-(n-k-1)d=(d+1)(k+1)-dn=\mu_k(X).\]
		Let \[Y=\mathrm{Join}\{X[\{v\}\cup M_{\sigma_0}(v)]:v\in V\setminus \sigma_0\}.\] 
		By Claim \ref{claim::1}, we see that $Y\cong (\Delta_d^{(d-1)})^{*(n-k-1)}$, and hence $\dim(Y)=d(n-k-1)-1=k-r$. Let $X'=Y*X[A]$. We have $X'\cong (\Delta_d^{(d-1)})^{*(n-k-1)}*\Delta_{r-1}$ and $\dim(X')=d(n-k-1)+r-1=k$. Furthermore, by Theorem \ref{thm::extremal_case},  $\mu_k(X')=r=\mu_k(X)$. Additionally, from Claim \ref{claim::1} and Claim \ref{claim::3} we immediately deduce that $X\subseteq X'$. Hence, $k\leq \dim(X)\leq \dim(X')=k$, and we have $ \dim(X)=\dim(X')=k$. 
		
		In what follows, we shall prove that $X=X'$. By contradiction, assume that $X$ is a proper subcomplex of $X'$. Note that,	for every simplicial complex $Z$ with $\dim(Z)=k$, we have $Z(k+1)=\varnothing$. By Lemma \ref{lem::Matrix_Rep}, the matrix representation $L_k(Z)$ of the reduced $k$-dimensional Laplacian of $Z$ is of the form
		\begin{equation*}
			L_k(Z)(\sigma,\tau)=\begin{cases} 
				k+1, & \mbox{if $\sigma=\tau$}, \\
				(\sigma:\sigma\cap\tau)\cdot(\tau:\sigma\cap\tau), &\mbox{if $|\sigma\cap\tau|=k$}, \\
				0, & \mbox{otherwise}.
			\end{cases}   
		\end{equation*}
		This implies that the value of  $L_k(Z)(\sigma,\tau)$ is determined by $\sigma$ and $\tau$ themselves. Consequently, we assert that  $L_k(X)$ is a principal submatrix of $L_k(X')$. Since $\lambda_{\min}(L_k(X))=\mu_k(X)=r=\mu_k(X')=\lambda_{\min}(L_k(X'))$, by  Lemma \ref{lem::principal}, we immediately deduce a contradiction if the following statement holds:
		
		\begin{enumerate}[(F4)]
			\item Every eigenvector of  $L_k(X')$ corresponding to $\mu_k(X')=r$ has no zero entries.
		\end{enumerate}
		
	Now we shall prove  (F4). Note that  $A \subseteq \sigma$ for every $\sigma \in X'(k)$. Since $\dim(X') = k$ and $|A|=r$, we have $Y(k-r)=\{\sigma \setminus A : \sigma \in X'(k)\}$. Moreover, we assert that $\deg_{Y}(\sigma \setminus A) = \deg_{X'}(\sigma) = 0$ for all $\sigma\in X'(k)$ because  $\dim(Y)=k-r$ and $\dim(X')=k$. Therefore, for any $\sigma,\tau\in X'(k)$,  $\sigma \setminus A \sim \tau \setminus A$ in $Y$ if and only if $\sigma \sim \tau$ in $X'$. We fix an order $\prec$ on $V(X')$ such that $u \prec v$ whenever $v \in A$ and $u \in V(X') \setminus A$. Then $(\sigma \setminus A : \tau \setminus A) = (\tau : \sigma)$ for any $\sigma, \tau \in X'(k)$. Consequently, from  Lemma \ref{lem::Matrix_Rep} we obtain
		\[ L_k(X') = L_{k-r}(Y) + rI. \]
This implies that $\mu_{k-r}(Y) = 0$, and the matrices $L_k(X')$ and $L_{k-r}(Y)$ share a common set of eigenvectors that  correspond to $\mu_{k}(X')$ and  $\mu_{k-r}(Y)$, respectively. Thus it suffices to prove that every eigenvector of  $L_{k-r}(Y)$ corresponding to $\mu_{k-r}(Y)$ has no zero entries. By contradiction, suppose that $\boldsymbol{y}=(y_\tau)_{\tau\in Y(k-r)}\in \mathbb{R}^{|Y(k-r)|}$ is an eigenvector of  $L_{k-r}(Y)$ corresponding to $\mu_{k-r}(Y)$  such that $y_{\tau_0}=0$ for some $\tau_0\in Y(k-r)$.  Let $Z=Y\setminus\{\tau_0\}$. We see that $Z$ is a $(k-r)$-dimensional proper subcomplex of $Y$. On the other hand, by Lemma \ref{lem::skeleton}, Lemma \ref{lem::SnSm} and Lemma \ref{lem::join}, we get    
		\[
		\|Y\|\cong\|(\Delta_d^{(d-1)})^{*(n-k-1)}\| \cong \|\Delta_d^{(d-1)}\|^{*(n-k-1)} 
		\cong (S^{d-1})^{*(n-k-1)} 
		\cong S^{d(n-k-1)-1} = S^{k-r},
		\]
		and hence $Y$ is a triangulation of the $(k-r)$-dimensional sphere $S^{k-r}$.  Since any proper subcomplex of a triangulation of an $(k-r)$-dimensional sphere has trivial $(k-r)$-dimensional cohomology (see, for example, \cite[p. 13]{Lew20}), we have $\widetilde{H}^{k-r}(Z;\mathbb{R})=0$, and by Theorem \ref{thm::Hodge},
		\begin{equation}\label{equ::subcomplex}
		\mu_{k-r}(Z)>0.
		\end{equation}
	Note that $\dim(Z) = \dim(Y)=k-r$. As above, we can show that $L_{k-r}(Z)$ is a proper principal submatrix of $L_{k-r}(Y)$ indexed by $Y(k-r)\setminus\{\tau_0\}$. Recall that $y_{\tau_0}=0$. By Lemma \ref{lem::principal}, we immediately deduce that \[\mu_{k-r}(Z)=\lambda_{\min}(L_{k-r}(Z))=\lambda_{\min}(L_{k-r}(Y))=\mu_{k-r}(Y)=0,\] contrary to  \eqref{equ::subcomplex}. This proves (F4).
		
		Therefore, we conclude that $X=X'\cong (\Delta_d^{(d-1)})^{*(n-k-1)}*\Delta_{r-1}=(\Delta_d^{(d-1)})^{*(n-k-1)}*\Delta_{(d+1)(k+1)-dn-1}$. In particular, $\dim(X)=\dim(X')=k$.
	\end{proof}

	\section*{Acknowledgements}
	
	The authors would like to thank Dr. Lu Lu for helpful discussions. X. Huang was supported by National Natural Science Foundation of China (Grant No. 11901540). H. Lin was supported by the National Natural Science Foundation of China (Nos. 12271162 and 12326372), and Natural Science Foundation of Shanghai (Nos. 22ZR1416300 and 23JC1401500)
	and The Program for Professor of Special Appointment (Eastern Scholar) at Shanghai Institutions of Higher Learning (No. TP2022031).	
	
	

\end{document}